\documentclass[pdf, 12pt]{article}
\usepackage{graphicx,amsmath,amssymb,amsthm}
\usepackage{epstopdf}
\usepackage{cite}

\input xy
\xyoption{all}

\def\E{\mathbb{E}}

\newtheorem{Remark}{Remark}[section]
\newtheorem{Assumption}{Assumption}[section]

\newtheorem{theorem}{Theorem}[section]

\newtheorem{lemma}[theorem]{Lemma}
\newtheorem{definition}[theorem]{Definition}

\newtheorem*{thm21}{Theorem 2.1}
\newtheorem*{thm22}{Theorem 2.2}
\newtheorem*{thm23}{Theorem 2.3}

\begin{document}
\title{Long-Run Accuracy of Variational Integrators in the Stochastic Context}

\author{Nawaf Bou-Rabee\thanks{Courant Institute of Mathematical Sciences, New York 
University, 251 Mercer Street, New York, NY 10012-1185 ({\tt nawaf@cims.nyu.edu}).}   \and 
Houman Owhadi\thanks{Applied \& Computational Mathematics and Control \& Dynamical 
Systems, Caltech, Pasadena, CA 91125 (\tt{owhadi@acm.caltech.edu}).}
}

\maketitle

\begin{abstract}

	This paper presents a Lie-Trotter splitting for inertial Langevin equations (Geometric Langevin Algorithm) 
	and analyzes its long-time statistical properties.  The splitting is defined as a composition of a 
	variational integrator with an Ornstein-Uhlenbeck flow.   Assuming the exact solution 
	and the splitting are geometrically ergodic, the paper proves the discrete invariant 
	measure of the splitting approximates the invariant measure of inertial Langevin 
	to within the accuracy of the variational integrator in representing the Hamiltonian.   
	In particular,  if the variational integrator admits no energy error, then the method samples 
	the invariant measure of inertial Langevin without error.  Numerical validation is provided 
	using explicit variational integrators with first, second, and fourth order accuracy.  
	
\end{abstract}


\medskip




\section{Introduction}

\paragraph{Overview}

This paper analyzes equilibrium statistical accuracy of discretizations of 
inertial Langevin equations based on variational integrators.   
Variational integrators are time-integrators adapted to the structure 
of mechanical systems\cite{MaWe2001}.   The theory of variational integrators includes 
discrete analogs of the Lagrangian, Noether's theorem, the Euler-Lagrange equations, 
and the Legendre transform.  Variational integrators can incorporate holonomic constraints 
(via, e.g., Lagrange multipliers) \cite{WeMa1997} and multiple time steps to obtain so-called 
asynchronous variational integrators \cite{LeMaOrWe2003}.

The generalization of variational integrators the paper analyzes are derived from a 
Lie-Trotter splitting of inertial Langevin equations into Hamiltonian and 
Ornstein-Uhlenbeck equations.   The integrator is then defined by selecting a variational 
integrator to approximate the Hamiltonian flow and using the exact Ornstein-Uhlenbeck flow.  
Such a generalization of variational integrators to inertial 
Langevin equations will be called a Geometric Langevin Algorithm (GLA).

This type of splitting of inertial Langevin equations is natural, but seems to have been only 
recently introduced in the literature (for molecular dynamics see \cite{VaCi2006, BuPa2007}, for 
dissipative particle dynamics see \cite{Sh2003, SeFaEsCo2006}, and for inertial particles see 
\cite{PaStZy2008}).   This paper is geared towards applications in molecular dynamics where 
inertial Langevin integrators (including the ones cited above) have been based on 
generalizations of the widely used St\"{o}rmer-Verlet integrator.   The St\"{o}rmer-Verlet 
integrator is attractive for molecular dynamics because it is an explicit, symmetric, second-order 
accurate, variational integrator for Hamilton's equations. In molecular dynamics it was 
popularized by Loup Verlet in 1967.   Other popular generalizations of the St\"{o}rmer-Verlet 
integrator to inertial Langevin equations include Br\"{u}nger-Brooks-Karplus (BBK) 
\cite{BrBrKa1984}, van Gunsteren and Berendsen (vGB) \cite{GuBe1982}, and the Langevin-
Impulse (LI) methods  \cite{SkIz2002}.  The LI method is also based on a splitting of inertial 
Langevin equations, but it is different from the splitting considered here.     To our knowledge 
there are few results in the literature which quantify the long-time statistical accuracy of the
Lie-Trotter splitting considered here.

GLA is not only quasi-symplectic as defined in RL1 and RL2 of 
\cite{MiTr2003}, but also conformally symplectic, i.e., preserves the precise symplectic area change 
associated to the flow of inertial Langevin processes \cite{McPe2001}.  One way to prove this property 
is by deriving the scheme from a variational principle and analyzing its boundary terms as done 
in the context of stochastic Hamiltonian systems without dissipation in \cite{BoOw2009A}.

\paragraph{Organization of the Paper}

In \S \ref{MainResults} the main results of the paper are presented.
\S \ref{Preliminaries} states all of the hypotheses used in the paper.  
These hypotheses are invoked in \S \ref{AnalysisOfGLA} where it is 
proved that GLA is pathwise convergent on finite time intervals (Theorem~\ref{GLAaccuracy}), 
GLA is geometrically ergodic with respect to a nearby invariant measure  on infinite time intervals
(Theorem~\ref{GLAgeometricergodicity}),  and the equilibrium statistical 
accuracy of GLA is governed by the order of accuracy of the variational integrator
in representing the Hamiltonian (Theorem~\ref{GLABGaccuracy}).  
In \S \ref{Validation}, numerical validation is provided.   In the Appendix 
we review some basic facts on variational integrators for the reader's 
convenience.

\paragraph{Limitations}

In a nutshell the main result of the paper states that if GLA is geometrically ergodic with 
respect to a unique invariant measure, the error in sampling the invariant measure of the 
SDE is determined by the energy error in GLA's variational integrator.
Now if the inertial Langevin equations have nonglobally
Lipschitz drift and the GLA is based on an explicit variational integrator, 
GLA may fail to be geometrically ergodic.  In particular, for any step-size there 
will be regions in phase space where the Lipschitz constant of the drift is beyond 
the linear stability threshold of GLA's underlying variational 
integrator.  Hence, an explicit GLA will be stochastically unstable.  Since
our results rely on a strong form of stochastic stability of GLA (namely,
geometric ergodicity), they may not hold in this case.

To stochastically stabilize GLA, one can use GLA as a proposal move in a 
Metropolis-Hasting method.  For a numerical analysis of the Metropolis-adjusted 
scheme, the reader is referred to \cite{BoVa2009A}.   A difficulty in Metropolizing
inertial Langevin is that its solution is not reversible.  
However, the solution composed with a momentum flip is reversible.  
The role of momentum flips in Metropolizing Langevin integrators is qualitatively 
and computationally analyzed in \cite{ScLeStCaCa2006,BuPa2007,AkBoRe2009}. 
For a quantitative treatment of the role of momentum flips in pathwise accuracy 
the reader is referred to \cite{BoVa2009A}.

\paragraph{Extension to manifolds}

For the sake of clarity, the setting of this paper is inertial Langevin equations on a flat space, 
but we stress GLA and its properties generalize to manifolds.
We refer to Remark \ref{RemarkGen} and to \cite{BoOw2009B}
for details.

\paragraph{Acknowledgements}

We wish to thank Christof Sch\"{u}tte and Eric Vanden-Eijnden for valuable advice.  
Denis Talay and Nicolas Champagnat helped sharpen the main result of the paper and put the 
paper in a better context.

This work was supported in part by DARPA DSO under AFOSR contract FA9550-07-C-0024.
N. B-R.~would like to acknowledge the support of the Berlin Mathematical School (BMS) and 
the United States National Science Foundation through NSF Fellowship \# DMS-0803095. 


\section{Main Results of Paper} \label{MainResults}

\paragraph{Inertial Langevin}

The setting of the paper is a dissipative stochastic Hamiltonian system 
(as in \cite{So1994, Ta2002})  on $\mathbb{R}^n$,  with phase space 
$\mathbb{R}^{2n}$, and smooth Hamilton $H \in C^{\infty}(\mathbb{R}^{2n}, \mathbb{R})$.
In terms of which consider the following inertial Langevin equations 
\begin{align} \label{InertialLangevin}
\begin{cases}
d \mathbf{Y} &= \mathbb{J} \nabla H(  \mathbf{Y} ) dt 
- \gamma \boldsymbol{C} \nabla H(  \mathbf{Y} ) dt 
+ \sqrt{2 \gamma \beta^{-1}} \boldsymbol{C}  d \mathbf{W} \\
\boldsymbol{Y}(0) &= \boldsymbol{x} \in \mathbb{R}^{2n}
\end{cases}
\end{align}
where the following matrices have been introduced:
\[
\mathbb{J}= 
\begin{bmatrix}
0 & \mathbf{I} \\
-\mathbf{I} &  0 
\end{bmatrix},~~~
\boldsymbol{C} = 
\begin{bmatrix}
0 & 0 \\
0 & \mathbf{I} 
\end{bmatrix}  \text{.}
\]
Here $\boldsymbol{W}$ is a standard $2n$-dimensional Wiener process, 
or Brownian motion, $\beta>0$ is a parameter referred to as the inverse 
temperature, and $\gamma>0$ is referred to as the friction factor.  We will often
write the continuous solution in component form as
$\mathbf{Y}(t)=(\boldsymbol{Q}(t),\boldsymbol{P}(t))$ where $\boldsymbol{Q}(t)$ 
and  $\boldsymbol{P}(t)$ represent the instantaneous configuration and 
momentum of the system, respectively.   We shall assume the Hamiltonian
is separable and quadratic in momentum:
\[
H(\boldsymbol{q}, \boldsymbol{p}) = 
\frac{1}{2}  \boldsymbol{p}^T \boldsymbol{M}^{-1} \boldsymbol{p} + U(\boldsymbol{q}) \text{,}
\]
where $\boldsymbol{M}$ is a symmetric positive definite mass matrix and 
$U$ is a potential energy function. 
Despite the degenerate diffusion in \eqref{InertialLangevin}, 
under certain regularity conditions on $U$, the solution to this SDE is geometrically 
ergodic with respect to an invariant probability measure $\mu$ with the following 
density\cite{Ta2002}:
\begin{equation} \label{BGdistribution}
\pi(\boldsymbol{q}, \boldsymbol{p}) = 
Z^{-1} \exp\left( -  \beta H(\boldsymbol{q}, \boldsymbol{p}) \right) \text{,}
\end{equation}
where $Z= 
\int_{\mathbb{R}^{2n}} 
\exp\left(-   \beta H(\boldsymbol{q}, \boldsymbol{p}) \right) 
d\boldsymbol{q} d\boldsymbol{p}$.   The invariant measure $\mu$ is known as
the {\em Boltzmann-Gibbs measure}.

\paragraph{Geometric Langevin Algorithm}

Let $N$ and $h$ be given, set $T= N h$ and $t_k = h k $ for
$k=0,...,N$.   Observe that the conservative part of \eqref{InertialLangevin} defines 
Hamilton's equations for the Hamiltonian $H$: \[
d \boldsymbol{Y} = \mathbb{J} \nabla H( \boldsymbol{Y} ) dt
\] or,
\begin{equation} \label{HamiltonsEquations}
\begin{cases} 
d \boldsymbol{Q} &= \boldsymbol{M}^{-1} \boldsymbol{P} dt  \\
d \boldsymbol{P}  &= - \nabla U(\boldsymbol{Q}) dt 
\end{cases}
\end{equation}
Let $h$ be a fixed step-size.  We apply a $pth$-order accurate variational integrator, 
$\theta_h: \mathbb{R}^{2n} \to \mathbb{R}^{2n}$, to approximate the Hamiltonian flow of \eqref{HamiltonsEquations} 
($p \ge 1$).  The nonconservative part of the inertial Langevin equation defines an 
Ornstein-Uhlenbeck process in momentum governed by the following linear SDE: \[
d \mathbf{Y} =
- \gamma \boldsymbol{C} \nabla H(  \mathbf{Y} ) dt 
+ \sqrt{2 \gamma \beta^{-1}} \boldsymbol{C}  d \mathbf{W}
\] or,
\begin{equation} \label{OrnsteinUhlenbeck}
\begin{cases} 
d \boldsymbol{Q} &= 0 \\
d \boldsymbol{P}  &= - \gamma \boldsymbol{M}^{-1} \boldsymbol{P} dt 
				 + \sqrt{2 \beta^{-1} \gamma} d \boldsymbol{W} 
\end{cases}
\end{equation}
Reference \cite{PaStZy2008} aptly refers to \eqref{OrnsteinUhlenbeck} as a Gaussian SDE since its 
stationary distribution on $\mathbb{R}^{2n}$ is Gaussian in momentum.

The following stochastic evolution map $\psi_{t_k+h, t_k} :  \mathbb{R}^{2n} \to \mathbb{R}^{2n}$ defines the 
stochastic flow of \eqref{OrnsteinUhlenbeck}:
\begin{align} \label{exactpsi}
& \psi_{t_k+h,t_k}:  \nonumber \\
& \qquad (\boldsymbol{q},\boldsymbol{p}) \mapsto  \left(\boldsymbol{q},  e^{-\gamma \boldsymbol{M}^{-1} h} \boldsymbol{p} 
+ \sqrt{2 \beta^{-1} \gamma} \int_{t_k}^{t_k+h} e^{-\gamma \boldsymbol{M}^{-1} (t_k+h-s)} d \boldsymbol{W}(s) \right) \text{,}  
\end{align}
with $\psi_{s,s}(\boldsymbol{x}) = \boldsymbol{x}$ and for  $0 \le r \le s \le t$ 
recall the Chapman-Kolmogorov identity $\psi_{t,s} \circ \psi_{s,r}(\boldsymbol{x}) = \psi_{t,r}(\boldsymbol{x})$
for all $\boldsymbol{x} \in \mathbb{R}^{2n}$.   For the distribution of the solution, the stochastic flow will be denoted simply by $\psi_{h}$.   
To make this map explicit, let  $\boldsymbol{\xi} \sim \mathcal{N}(\boldsymbol{0},\boldsymbol{I})$ and set
\begin{align*}
 \boldsymbol{\Sigma}_h :=&   
 2 \beta^{-1}  \gamma \E\left\{ \left( \int_0^h e^{-\gamma \boldsymbol{M}^{-1} (h-s)} d \boldsymbol{W}(s) \right) 
 \left( \int_0^h e^{-\gamma \boldsymbol{M}^{-1} (h-s)} d \boldsymbol{W}(s) \right)^T \right\} \\
 =& \beta^{-1} \left( \boldsymbol{I}  - \exp(- 2 \gamma \boldsymbol{M}^{-1} h) \right) \boldsymbol{M}  
\end{align*}
and define $A_h$ to be the decomposition matrix arising from the Cholesky factorization 
of $\boldsymbol{\Sigma}_{h}$, i.e., $\boldsymbol{A}_{h} \boldsymbol{A}_{h}^T = \boldsymbol{\Sigma}_{h}$.
In terms of these, introduce the following flow map:
\begin{align} \label{psi}
 \psi_h:  (\boldsymbol{q},\boldsymbol{p}) \mapsto  
 \left(\boldsymbol{q},  e^{-\gamma \boldsymbol{M}^{-1} h} \boldsymbol{p} + A_h \boldsymbol{\xi} \right) \text{.}  
\end{align}
In distribution \eqref{psi} is identical to \eqref{exactpsi}.

Given $\boldsymbol{X}_k \in \mathbb{R}^{2n}$ and $h$, the {\em Geometric Langevin Algorithm} (GLA) is defined as
the following Lie-Trotter splitting integrator for \eqref{InertialLangevin}:
\begin{equation} \label{GLA}
\boldsymbol{X}_{k+1} := \theta_h \circ \psi_{t_k+h,t_k} (\boldsymbol{X}_k) 
\end{equation}
for $k=0,...,N-1$ with $\boldsymbol{X}_0 = \boldsymbol{x}$.

\begin{Remark}\label{RemarkGen}
Observe that GLA generalizes to inertial Langevin equations on a manifold.
This generalization is possible because 
its symplectic component can be defined as a variational integrator for Hamilton's equations
on a manifold and its Ornstein-Uhlenbeck component can be defined as the solution of
an SDE on a vector space.  This generalization is motivated by molecular systems
with holonomic constraints.  As mentioned in the introduction, variational integrators can
incorporate holonomic constraints.   In the special case that the configuration manifold of GLA
is compact (e.g., $SO(3)$) and the potential energy is smooth, then the assumption on the 
geometric ergodicity of GLA is typically satisfied for sufficiently small time-step.
\end{Remark}

Given $\boldsymbol{Z}_k \in \mathbb{R}^{2n}$ and $h$, let $\vartheta_h: \mathbb{R}^{2n} \to \mathbb{R}^{2n}$ denote
the exact time-$h$ flow of Hamilton's equations \eqref{HamiltonsEquations}.  
The {\em Exact Splitting} is defined as
\begin{equation}  \label{ExactSplitting}
\boldsymbol{Z}_{k+1} := \vartheta_h \circ \psi_{t_k+h, t_k} ( \boldsymbol{Z}_k )
\end{equation}
for $k=0,...,N-1$ with $\boldsymbol{Z}_0 = \boldsymbol{x}$.

\paragraph{Properties of GLA}

The assumptions that appear in the following theorems are provided in \S \ref{Preliminaries}.

Let $\E^{\boldsymbol{x}}\{ \cdot \}$ denote the expectation conditioned on the initial 
condition being $\boldsymbol{x} \in \mathbb{R}^{2n}$.  In terms of this notation, we can quantify 
the strong convergence of GLA  to solution trajectories of  inertial Langevin 
\eqref{InertialLangevin}.  The precise statement follows
\begin{theorem} [Pathwise Accuracy] \label{GLAaccuracy}
Assume~\ref{sa1} and \ref{sa2}.  
For any $T>0$, there exist $h_c>0$ and $C(T)>0$, such that for all $h<h_c$, 
$\boldsymbol{x} \in \mathbb{R}^{2n}$, and $t\in[0,T]$, GLA satisfies
\begin{equation} 
( \E^{\boldsymbol{x}} \{ | \boldsymbol{X}_{\lfloor t/h \rfloor} - \boldsymbol{Y}( \lfloor t/h \rfloor h) |^2 \} )^{1/2}  \le 
C(T) (1+|\boldsymbol{x}|^2)^{1/2} h \text{.} 
\end{equation}
\end{theorem}
This result is expected because a Lie-Trotter splitting is first-order for deterministic ODEs, 
and the noise in \eqref{InertialLangevin} is additive.

Using this pathwise convergence,  it is shown that GLA is geometrically ergodic with respect to
a discrete invariant measure $\mu_h$.
\begin{theorem} [Geometric Ergodicity] \label{GLAgeometricergodicity}
Assume \ref{sa1}, \ref{sa2}, and \ref{sa3}.   Then GLA is geometrically ergodic with respect 
to a discrete invariant measure $\mu_h$ and the continuous Lyapunov
function (cf.~Assumption~\ref{sa3}).  That is, there exist $h_c>0$,  $\lambda>0$ ,  and 
$C_3 > 0$, such that for all $h<h_c$ and for all $k \ge 2$, 
\[
| \E^{\boldsymbol{x}} \left\{ f( \boldsymbol{X}_k ) \right\} - \mu_h(f) | 
\le C_3 V(\boldsymbol{x}) e^{-\lambda k h}, ~~\forall~\boldsymbol{x} \in \mathbb{R}^{2n},
\]
and for all  test functions satisfying $| f(\boldsymbol{y}) | \le C_3 V(\boldsymbol{y})$
 for all $\boldsymbol{y} \in \mathbb{R}^{2n}$.
\end{theorem}
We stress this result is a consequence of strong convergence of GLA and the 
assumptions made on the potential energy and variational integrator.   These assumptions
are sufficient, but not necessary to guarantee this result.

Using geometric ergodicity we can quantify the equilibrium statistical accuracy of GLA.
If $p$ represents the global accuracy of GLA's underlying variational integrator, then 
$\mu_h$ is in TV distance $O(h^p)$ away from the Boltzmann-Gibbs measure $\mu$.   
To be precise, the main result of the paper states
\begin{theorem} [Long-Run Accuracy] \label{GLABGaccuracy}
Assume \ref{sa1}, \ref{sa2}, and \ref{sa3}.  
Let $\mu_h$ denote the discrete invariant measure of GLA. 
Then, there exist $C>0$ and $h_c>0$, such that for all $h<h_c$, 
\[
  | \mu - \mu_h |_{TV} \le C h^{p}  \text{.}
\]
\end{theorem}
There is a stronger argument in \cite{Ta2002} based on the Feynman-Kac formula that can extend
Theorem~\ref{GLABGaccuracy} to
\begin{equation}
| \mu(f) - \mu_h(f) | \le C h^p \text{,}
\end{equation}
for all test functions $f \in L^2_{\mu}(\mathbb{R}^{2n})$ that are smooth with polynomial growth at infinity.  
The paper proves Theorem~\ref{GLABGaccuracy} with a more direct strategy.  An important point
is that the proof is transparent since it involves a forward error analysis and does not rely on 
knowing the precise form of $\mu_h$.  The proof relies on the existence of $\mu_h$ and
the nature of the convergence of GLA from a nonequilibrium position.  Indeed, a backward error analysis 
of this discretization of the SDE \eqref{InertialLangevin} to characterize this invariant measure would 
be substantially  more involved.

\paragraph{Implications}

As a consequence of the TV error estimate derived in this paper, one can control the order of 
accuracy of  $\mu_h$  by controlling the order of accuracy of  GLA's underlying variational 
integrator.    This is the distinguishing feature of GLA. Existing theory would indicate the 
accuracy of $\mu_h$ is the same order as the weak or strong accuracy of GLA.   
Theorem~\ref{GLAaccuracy} states GLA is just first-order accurate on solution trajectories.
Hence, existing theory would suggest that the equilibrium statistical accuracy of GLA is
first-order, rather than $pth$-order accurate (where $p$ is the order of accuracy of GLA's
underlying variational integrator).

Existing theory would indicate to obtain a higher-order approximation of the invariant measure 
one would require a higher-order approximant to SDE \eqref{InertialLangevin}  which entails approximation of multiple 
$n$-dimensional stochastic integrals per time-step.  It is well-known that such higher-order 
discretizations of SDEs are computationally intensive.  In contrast, a step of 
GLA requires evaluation of a single, $n$-dimensional stochastic integral per time-step.   
According to the main result of this paper, the order of accuracy of the  
variational integrator can be used to tune the TV-distance in Theorem~\ref{GLABGaccuracy} 
to a desired tolerance.


\section{Preliminaries}  \label{Preliminaries}

The following assumptions on the potential energy, $U: Q \to \mathbb{R}$, will be used in this 
paper.  These hypotheses are the same as those made in \S 7 of  \cite{MaStHi2002}.

\begin{Assumption}[Assumptions on Potential Energy] \label{sa1}
The potential energy function $U \in C^{\infty}(\mathbb{R}^n, \mathbb{R})$ satisfies:
\begin{description}
\item[U1)] there exists a real constant $A_0 > 0$ such that  
\[
| \nabla U(\boldsymbol{q}_0) - \nabla U(\boldsymbol{q}_1)| \le 
A_0 | \boldsymbol{q}_0 - \boldsymbol{q}_1 | \text{,}~~\forall~\boldsymbol{q}_0 , \boldsymbol{q}_1 \in \mathbb{R}^n \text{.}
\]
\item[U2)] there exists a real constant $A_1 > 0$ such that
\[
U( \boldsymbol{q} )  \ge A_1 (1 + | \boldsymbol{q} | ^2 ) \text{,}  ~~\forall~ \boldsymbol{q} \in \mathbb{R}^n \text{.}
\]
\end{description}
\end{Assumption}

By standard results in stochastic analysis, condition {\bf U1} is sufficient to guarantee almost 
sure existence and pathwise uniqueness  of a solution to \eqref{InertialLangevin}.   The condition 
{\bf U2} ensures that $e^{-\beta H}$ is integrable over $\mathbb{R}^{2n}$, and hence, that the Boltzmann-Gibbs 
measure is a well-defined probability measure. Assuming the solution to \eqref{InertialLangevin} is 
geometrically ergodic, we will prove in this paper that conditions {\bf U1} - {\bf U2} together with the 
following assumptions on the variational integrator, $\theta_h: \mathbb{R}^{2n} \to \mathbb{R}^{2n}$, are sufficient 
(but not necessary) to guarantee geometric ergodicty of GLA.

\begin{Assumption}[Assumptions on Variational Integrator] \label{sa2}
For any $t>0$ let $\vartheta_t$ denote the exact Hamiltonian flow of \eqref{HamiltonsEquations}.
The variational integrator $\theta_h : \mathbb{R}^{2n} \to \mathbb{R}^{2n}$ satisfies the following.
\begin{description}
\item [V1)] $\theta_h$ is the discrete Hamiltonian map of a hyperregular discrete Lagrangian 
$L_d: \mathbb{R}^n \times \mathbb{R}^n \to \mathbb{R}$  (cf.~\eqref{HyperRegularity} and \cite{MaWe2001}).
\item[V2)] there exist constants $B_0>0$  and $h_c>0$, such that for any $h<h_c$,  
\[
| \theta_{h}(\boldsymbol{x}) - \vartheta_{h}(\boldsymbol{x}) | 
\le B_0 (1 + | \boldsymbol{x} |^2 )^{1/2} h^{p+1},~~\forall~\boldsymbol{x} \in \mathbb{R}^{2n} \text{.}
\]
\end{description}
\end{Assumption}

As discussed in Appendix I, the condition {\bf V1}  implies that $\theta_h$ is symplectic, and 
hence, Lebesgue measure preserving.  It will also be an important ingredient in proving 
Theorem~\ref{GLAgeometricergodicity} on geometric ergodicity of GLA.   The condition
{\bf V2} states that the integrator is locally $(p+1)th$-order accurate.

Finally, we make the following structural assumption on \eqref{InertialLangevin}.
\begin{Assumption}[Existence of a Lyapunov Function] \label{sa3}
There exists $V \in C^{\infty}( \mathbb{R}^{2n} , \mathbb{R})$ and constants $C_i >0$ such that 
\[
C_0 ( 1+   | \boldsymbol{x} |^2) \le V(\boldsymbol{x}) \le C_1 (1 + | \boldsymbol{x} |^2),~~
\nabla V(\boldsymbol{x}) \le C_2 (1+ | \boldsymbol{x} | ), ~~
\forall~\boldsymbol{x} \in \mathbb{R}^{2n},
\]
$\lim_{\boldsymbol{x} \to \infty} V(\boldsymbol{x}) = \infty$, $a>0$ and $c>0$, such that
for all $t>0$,
\[
\E^{\boldsymbol{x}} \{ V( \boldsymbol{Y}(t) ) \} \le 
e^{-a t } V( \boldsymbol{x} ) + \frac{c}{a} (1 - e^{-a t} ),~~\forall~\boldsymbol{x} \in \mathbb{R}^{2n} \text{.}
\]
\end{Assumption}


\section{Analysis of GLA}   \label{AnalysisOfGLA}

\subsection{Pathwise Convergence}

Here GLA is shown to be first-order mean-squared convergent, 
which is a notion of pathwise convergence to solutions of \eqref{InertialLangevin} \cite{Ta1995,MiTr2004}.  
The first-order accuracy of GLA on solution trajectories is not surprising because the 
method is derived from a Lie-Trotter splitting of \eqref{InertialLangevin}.  
It is simply a generalization of the well-known fact that Lie-Trotter splittings of deterministic 
ODEs yield first-order accurate methods.  This
generalization is possible despite the lack of regularity in solutions because the noise in 
\eqref{InertialLangevin} is additive.  Since the proof is standard, it will be kept terse.

\begin{thm21}[Pathwise Accuracy] 
Assume~\ref{sa1} and \ref{sa2}.
For any $T>0$, there exist $h_c>0$ and $C(T)>0$, such that for all $h<h_c$, 
$\boldsymbol{x} \in \mathbb{R}^{2n}$, and $t\in[0,T]$,
\begin{equation} 
( \E^{\boldsymbol{x}} \{ | \boldsymbol{X}_{\lfloor t/h \rfloor} - \boldsymbol{Y}( \lfloor t/h \rfloor h ) |^2 \} )^{1/2}  
\le C(T) (1+|\boldsymbol{x}|^2)^{1/2} h \text{.} 
\end{equation}
\end{thm21}

\begin{proof}
By standard results in stochastic analysis, condition {\bf U1} 
guarantees there a.s. exists a pathwise unique solution to \eqref{InertialLangevin}: $
\mathbf{Y}(t) \in \mathbb{R}^{2n}$ for $t \in [0, T]$ with $\mathbf{Y}(0) = \boldsymbol{x}$.   
Moreover, one can obtain the following bound on the second moment of of 
the solution: for all $T>0$, there exists a $C(T)>0$ such that for all $t \in [0, T]$,
\begin{equation} \label{Ymomentbound}
\E^{\boldsymbol{x}} \left\{ |  \boldsymbol{Y}(t)  |^2 \right\} \le C(T) (1 + | \boldsymbol{x} |^2) \text{.}
\end{equation}

We will use this bound to invoke Theorem~1.1 in \cite{MiTr2004} which enables one 
to deduce global mean-squared error estimates of a discretization from local mean-squared error 
and local mean deviation.    First, we establish this estimate for the exact splitting \eqref{ExactSplitting}.   
By using Assumption {\bf U1}, it is straightforward to show (see Lemma~\ref{singlesteperror}) 
that there exists $C>0$ such that:
\begin{equation} 
| \E^{\boldsymbol{x}} \{   \boldsymbol{Y}(h)  - \boldsymbol{Z}_{1}  \} | 
\le C \left( 1+ | \boldsymbol{x} |^2 \right)^{1/2}  h^{2} 
\end{equation}
and
\begin{equation}  \label{YZmserror}
\left( \E^{\boldsymbol{x}} \{ | \boldsymbol{Y}(h)  - \boldsymbol{Z}_{1} |^2 \} \right)^{1/2} 
\le C \left( 1+ | \boldsymbol{x} |^2 \right)^{1/2}  h^{3/2} 
\end{equation}
Together with \eqref{Ymomentbound} this implies there exist $h_c>0$ and $C(T)>0$,
such that for all $h<h_c$,  $t \in [0, T]$ and $\boldsymbol{x} \in \mathbb{R}^{2n}$:
\begin{equation} \label{Zmomentbound}
 \E^{\boldsymbol{x}} \{ | \boldsymbol{Z}_{\lfloor t/h \rfloor} |^2 \} 
\le C(T) ( 1 + | \boldsymbol{x} |^2 )
\end{equation}
Hence, by Theorem~1.1 in \cite{MiTr2004}, one can show that for all $T>0$,
there exist $h_c>0$ and $C(T)>0$, such that for all $h<h_c$,  $t \in [0, T]$ and $\boldsymbol{x} \in \mathbb{R}^{2n}$:
\begin{equation} \label{YZglobalmserror}
( \E^{\boldsymbol{x}} \{  |  \boldsymbol{Z}_{\lfloor t/h \rfloor} - \boldsymbol{Y}(\lfloor t/h \rfloor h)  |^2 \} )^{1/2}  
\le C(T) \left( 1+ | \boldsymbol{x} |^2 \right)^{1/2}  h
\text{.}
\end{equation}

Observe that the difference between a single step of 
GLA \eqref{GLA} and the exact splitting \eqref{ExactSplitting} can be written as 
\[
  \boldsymbol{X}_{1}  - \boldsymbol{Z}_{1}  =   
  (\theta_h - \vartheta_h ) \circ \psi_{h,0}(\boldsymbol{x})   \text{.}
\]
Using Assumption {\bf V2} one can show there exists $C>0$ such that
\begin{equation} \label{XZmserror}
\left( \E^{\boldsymbol{x}} \{ |  \boldsymbol{X}_{1}  - \boldsymbol{Z}_{1} |^2 \} \right)^{1/2} 
\le C \left( 1 + | \boldsymbol{x} |^2  \right)^{1/2} h^{p+1} \text{,}
\end{equation}
and, by Jensen's inequality:
\begin{equation} \label{XZmerror}
| \E^{\boldsymbol{x}} \{   \boldsymbol{X}_{1}  - \boldsymbol{Z}_{1}  \} |  
\le C \left( 1+ | \boldsymbol{x} |^2 \right)^{1/2}  h^{p+1} 
\text{.}
\end{equation}
Together with \eqref{Zmomentbound} this implies that there exist $h_c>0$ and $C(T)>0$,
such that for all $h<h_c$,  $t \in [0, T]$ and $\boldsymbol{x} \in \mathbb{R}^{2n}$:
\begin{equation} \label{GLAmomentbound}
 \E^{\boldsymbol{x}} \{ | \boldsymbol{X}_{\lfloor t/h \rfloor} |^2 \} 
 \le C(T) ( 1 + | \boldsymbol{x} |^2 )
\end{equation}
Using Assumption {\bf U1} and Theorem~1.1 of \cite{MiTr2004}, one can also show that for all $T>0$,
there exist $h_c>0$ and $C(T)>0$, such that for all $h<h_c$,  $t \in [0, T]$ and $\boldsymbol{x} \in \mathbb{R}^{2n}$:
\begin{equation} \label{XZglobalmserror}
( \E^{\boldsymbol{x}} \{  | \boldsymbol{X}_{\lfloor t/h \rfloor}  - \boldsymbol{Z}_{\lfloor t/h \rfloor}  |^2 \} )^{1/2}  
\le C(T) \left( 1+ | \boldsymbol{x} |^2 \right)^{1/2}  h^{p} 
\text{.}
\end{equation}
In other words, GLA is $O(h^p)$ strongly convergent to the exact splitting.
One can then use the triangle inequality to obtain the estimate in the theorem from 
\eqref{XZglobalmserror} and \eqref{YZglobalmserror}, i.e.,
\begin{align*}
& ( \E^{\boldsymbol{x}} \{ | \boldsymbol{X}_{\lfloor t/h \rfloor} - \boldsymbol{Y}(\lfloor t/h \rfloor h) |^2 \} )^{1/2} \le \\
& \qquad \underset{\le K(T) \left( 1+ | \boldsymbol{x} |^2 \right)^{1/2}  h^{p}  }{\underbrace{ 
( \E^{\boldsymbol{x}} \{ | \boldsymbol{X}_{\lfloor t/h \rfloor} - \boldsymbol{Z}_{\lfloor t/h \rfloor} |^2 \} )^{1/2} }}  + 
 \underset{ \le K(T) \left( 1+ | \boldsymbol{x} |^2 \right)^{1/2}  h }{\underbrace{ 
( \E^{\boldsymbol{x}} \{ | \boldsymbol{Z}_{\lfloor t/h \rfloor} - \boldsymbol{Y}(\lfloor t/h \rfloor h) |^2 \} )^{1/2} }} \text{.}
\end{align*}
In sum, GLA is first-order strongly convergent to solutions of \eqref{InertialLangevin}.
\end{proof}

\subsection{Geometric Ergodicity}

Geometric ergodicity is a strong type of stochastic stability of a Markov chain \cite{MeTw2009}.
In this section geometric ergodicity of GLA is established  following the recipe provided 
in \S 7 of \cite{MaStHi2002}.  In the context of this paper, geometric ergodicity means,

\begin{definition}
A Markov chain $\boldsymbol{X}_k$ is said to be geometrically ergodic if there exist probability
measure $\mu_{\infty}$, $ \rho < 1$, and $M \in C^{\infty}(\mathbb{R}^{2n}, \mathbb{R}^+)$, 
such that \begin{equation} \label{eq:geometricergodicity}
| \E^{\boldsymbol{x}} \left\{ f( \boldsymbol{X}_k ) \right\} - \mu_{\infty}(f) |
\le M(\boldsymbol{x}) \rho^k,~~\forall~\boldsymbol{x} \in \mathbb{R}^{2n},~~\forall~k\in\mathbb{N} \text{,}
\end{equation}
and for all  $f \in L^2_{\mu_{\infty}}(\mathbb{R}^{2n})$ satisfying $| f(\boldsymbol{y}) | \le M(\boldsymbol{y})$ for all 
$\boldsymbol{y} \in \mathbb{R}^{2n}$.
\end{definition}

Under the hypotheses below, the Lyapunov function from Assumption~\ref{sa3} is inherited by GLA.

\begin{thm22} [Geometric Ergodicity]
Assume \ref{sa1}, \ref{sa2}, and \ref{sa3}.   Then GLA is geometrically ergodic with respect 
to a discrete invariant measure $\mu_h$ and the continuous Lyapunov
function (cf.~Assumption~\ref{sa3}).  That is, there exist $h_c>0$,  $\lambda>0$ ,  and 
$C_3 > 0$, such that for all $h<h_c$ and for all $k \ge 2$, 
\[
| \E^{\boldsymbol{x}} \left\{ f( \boldsymbol{X}_k ) \right\} - \mu_h(f) | 
\le C_3 V(\boldsymbol{x}) e^{-\lambda k h}, ~~\forall~\boldsymbol{x} \in \mathbb{R}^{2n},
\]
and for all  test functions satisfying $| f(\boldsymbol{y}) | \le C_3 V(\boldsymbol{y})$
 for all $\boldsymbol{y} \in \mathbb{R}^{2n}$.
\end{thm22}

\begin{proof}
This proof is an application of Theorem 2.5 of \cite{MaStHi2002}.   To invoke this theorem, we 
will show that GLA  inherits the Lyapunov function $V:\mathbb{R}^{2n} \to \mathbb{R}$ 
of the continuous solution (cf.~Assumption~\ref{sa3}) and satisfies 
a minorization condition when sampled every other step.

To prove that GLA inherits the Lyapunov function  $V:\mathbb{R}^{2n} \to \mathbb{R}$ we 
use Theorem 7.2 of \cite{MaStHi2002}.  This theorem assumes that the Lyapunov function of the SDE
is essentially quadratic which follows from Assumption~\ref{sa3}, and 
that the discretization of the SDE satisfies Condition 7.1 of \cite{MaStHi2002}.  
Condition 7.1 (i) is a consequence of
a single-step mean-squared error estimate of GLA which can be derived from 
\eqref{YZmserror} and \eqref{XZmserror}.  Condition 7.1 (ii) is satisfied for the first 
and second moments of GLA due to the estimate \eqref{GLAmomentbound}.  
Hence all of the assumptions of Theorem 7.2 \cite{MaStHi2002} are satisfied, and one can conclude 
that GLA inherits the Lyapunov  function $V:\mathbb{R}^{2n} \to \mathbb{R}$ up to a 
constant pre-factor.

Next, we prove that GLA satisfies a minorization 
condition when sampled every other step.  This property follows from Lemma 2.3 of 
\cite{MaStHi2002}, because GLA sampled every other step admits a strictly positive, smooth transition 
probability function.  In fact, this transition probability $q_h: \mathbb{R}^{2n} \times \mathbb{R}^{2n} \to [0, 1]$ 
can be explicitly characterized, and by inspection it is clear that it is 
smooth as a function of its arguments and strictly positive everywhere.

To derive this expression, let  $o_h :\mathbb{R}^n \times \mathbb{R}^n \to [0,1]$ denote the 
transition probability  of the Ornstein-Uhlenbeck flow $\psi_{h}$ \eqref{psi}.
 By a change of variables,  it's transition density is given explicitly by:
\begin{align}  \label{oh}
& o_h  (\boldsymbol{p}_0, \boldsymbol{p}_1) = \nonumber \\    
& \frac{1}{( 2 \pi )^{n/2}  | \det( \boldsymbol{\Sigma}_h) | }  
\exp\left(-\frac{1}{2} \left( \boldsymbol{p}_1 - e^{- \gamma \boldsymbol{M}^{-1} h} \boldsymbol{p}_0 \right)^T  
\boldsymbol{\Sigma}_h^{-1} \left( \boldsymbol{p}_1 - e^{- \gamma \boldsymbol{M}^{-1} h} \boldsymbol{p}_0 \right) \right)  
\text{,}
\end{align}
where
\[
 \boldsymbol{\Sigma}_h = \beta^{-1} \left( \boldsymbol{Id}  - \exp(- 2 \gamma \boldsymbol{M}^{-1} h) \right) \boldsymbol{M}  \text{.}
\]
Let $\mathbf{D} = \mathbb{R}^{2n} \times \mathbb{R}^{2n}  \times \mathbb{R}^{2n}  \times \mathbb{R}^n$.   
Since the maps $\theta_h$ and $\psi_h$ enjoy the Markov property, the transition probability of the composition 
$ \theta_h \circ \psi_h \circ \theta_h \circ \psi_h$ can be expressed as a product of 
the transition probabilities of its components:
\begin{align*}
&q_{h}  \left( (\boldsymbol{q},\boldsymbol{p}),  (\bar{\boldsymbol{q}},\bar{\boldsymbol{p}})  \right) =  \\
& \int_{\mathbf{D}}  o_h(\boldsymbol{p}, \boldsymbol{p}_1)  
\delta((\boldsymbol{q}_1, \boldsymbol{p}_2)-\theta_h(\boldsymbol{q},\boldsymbol{p}_1)) 
o_h (\boldsymbol{p}_2 , \boldsymbol{p}_3)  
\delta((\bar{\boldsymbol{q}}, \bar{\boldsymbol{p}})-\theta_h(\boldsymbol{q}_1,\boldsymbol{p}_3)) 
d \boldsymbol{p}_1 d \boldsymbol{p}_2 d \boldsymbol{p}_3 d \boldsymbol{q}_1 \text{.}
\end{align*}
The zero of the argument of the second Dirac-delta measure (from left) occurs 
at $(\boldsymbol{q}_1,\boldsymbol{p}_3) = \theta_h^{-1}(\bar{\boldsymbol{q}}, \bar{\boldsymbol{p}})$.  
Hence, the above expression simplifies,
\begin{align*}
& q_{h}  \left( (\boldsymbol{q},\boldsymbol{p}),  (\bar{\boldsymbol{q}},\bar{\boldsymbol{p}})  \right) =\\
& \int_{\mathbb{R}^{2n}  \times \mathbb{R}^{2n} }  o_h(\boldsymbol{p}, \boldsymbol{p}_1)  
\delta((\boldsymbol{q}_1, \boldsymbol{p}_2)-\theta_h(\boldsymbol{q},\boldsymbol{p}_1)) 
o_h (\boldsymbol{p}_2 , \boldsymbol{p}_3)  
d \boldsymbol{p}_1 d \boldsymbol{p}_2  \text{.}
\end{align*}
By condition {\bf V1} on $\theta_h$, the zero of the argument of the remaining Dirac-delta measure 
above is uniquely determined by the discrete Hamiltonian flow of
the discrete Lagrangian (cf.~\eqref{del} in Appendix I).  Hence, one obtains:
\begin{align} \label{qh}
 & q_{h}  \left( (\boldsymbol{q},\boldsymbol{p}),  (\bar{\boldsymbol{q}},\bar{\boldsymbol{p}})  \right)  =\nonumber \\ 
 & \quad | \det(D_{12} L_d(\boldsymbol{q}, \boldsymbol{q}_1,h) )|  
 o_h(\boldsymbol{p},  -D_1 L_d(\boldsymbol{q}, \boldsymbol{q}_1, h))  
 o_h( D_2 L_d(\boldsymbol{q}, \boldsymbol{q}_1, h), \boldsymbol{p}_3)   
\end{align}
where $(\boldsymbol{q}_1,\boldsymbol{p}_3) = \theta_h^{-1}(\bar{\boldsymbol{q}}, \bar{\boldsymbol{p}})$.  
Using the hyperregularity assumption on the variational integrator {\bf V1} (cf.~\eqref{HyperRegularity}),
 \eqref{oh}, and \eqref{qh}, it is clear that $q_h$ is a smooth probability transition function that 
is everywhere strictly positive.   Hence, by Lemma 2.3 of \cite{MaStHi2002}, GLA sampled every other step
satisfies a minorization condition.

In sum, we have shown that GLA satisfies a minorization 
condition and admits a Lyapunov function.  
The result follows from invoking Theorem 2.5 in \cite{MaStHi2002}.
\end{proof}

\subsection{Long-Run Accuracy}

Now we quantify the accuracy of GLA in sampling from the equilibrium measure of
\eqref{InertialLangevin}.  For this purpose recall the following definition.

\begin{definition}[Invariance of Measure]
A Markov chain $\boldsymbol{X}_k \in \mathbb{R}^{2n}$ is said to preserve a probability measure $\mu_{\infty}$ 
if for all $f \in L^2_{\mu_{\infty}}(\mathbb{R}^{2n})$ and $k \in \mathbb{N}$,
\begin{equation}  \label{eq:muinvariancecondition}
\E_{\mu_{\infty}} \E^{\boldsymbol{x}} \{ f( \boldsymbol{X}_k ) \} = \mu_{\infty}(f) 
\end{equation}
where $\mu_{\infty}(f) =   \int_{\mathbb{R}^{2n}} f    d \mu_{\infty}$ and $\E_{\mu_{\infty}} \E^{\boldsymbol{x}}$ denotes
expectation conditioned on the initial distribution being 
sampled from $\mu_{\infty}$, i.e.,
\[
\mathbb{E}_{\mu_{\infty}} \mathbb{E}^{\boldsymbol{x}} \left\{ f( \boldsymbol{X}_k ) \right\} = 
\int_{\mathbb{R}^{2n}} \mathbb{E}^{\boldsymbol{x}} \left\{ f( \boldsymbol{X}_k ) \right\}  
\mu_{\infty} (d \boldsymbol{x}) \text{.}
\]
\end{definition}

Given a step-size $h$, define the deviation GLA makes in preserving the Boltzmann-Gibbs measure, $\mu$, 
as $\Delta_h^k: L^2_{\mu}(\mathbb{R}^{2n}) \to \mathbb{R}$:
\[
\Delta_h^k(f) :=   \E_{\mu} \E^{\boldsymbol{x}} \{ f( \boldsymbol{X}_k ) \} -  \mu(f)
\text{.}
\]
Observe that if GLA exactly preserves $\mu$ then:
\[
\Delta_h^k(f) = 0, ~~~ \forall ~f \in L^2_{\mu}(\mathbb{R}^{2n}) \text{.}
\]
The following local error result follows from
the Ornstein-Uhlenbeck flow $\psi_h$ preserving $\mu$ and 
the variational integrator $\theta_h$ preserving Lebesgue measure.

\begin{lemma} \label{Deltah1}
Suppose the potential energy satisfies {\bf U2}.   For a given $f \in L^2_{\mu}(\mathbb{R}^{2n})$,  
\[
\Delta_h^1(f) = \int_{\mathbb{R}^{2n}} f(\boldsymbol{x}) 
\left( e^{-\beta \left(  H((\theta_h)^{-1} (\boldsymbol{x})) - H(\boldsymbol{x}) \right)} - 1  \right) 
\mu(d \boldsymbol{x}) \text{.}
\]
\end{lemma}

\begin{proof}
The condition {\bf U2} ensures that $\mu$ is a well-defined probability measure.
According to the definition of GLA \eqref{GLA}, $\boldsymbol{X}_1=  \theta_h \circ \psi_h(\boldsymbol{x}) $. 
Substitute this expression into $\Delta_h^1$ to obtain:
\[ 
\Delta_h^1(f)  = 
\int_{\mathbb{R}^{2n}} \E^{\boldsymbol{x}} \left\{ f (\theta_h \circ \psi_h(\boldsymbol{x}) ) \right\} \mu( d \boldsymbol{x} )  
-   \int_{\mathbb{R}^{2n}} f    d \mu \text{.}
\]
Since $\psi_h$ preserves $\mu$ and $\theta_h$ is deterministic it follows that,
\[ 
\Delta_h^1(f)  = 
\int_{\mathbb{R}^{2n}}  f (\theta_h(\boldsymbol{x}) )  \mu(d \boldsymbol{x} )  -   \int_{\mathbb{R}^{2n}} f    d \mu \text{.}
\]
Changing variables under the map $\theta_h$ in the first integral above, and using the 
volume-preserving property of the variational integrator $\theta_h$ (See Appendix.) 
one obtains the desired expression.
\end{proof}

\begin{Remark}
As a consequence of Lemma~\ref{Deltah1}, if $\theta_h$ admits no energy error, then
GLA preserves $\mu$.  In particular, the exact splitting \eqref{ExactSplitting} preserves $\mu$.
\end{Remark}

In the situation where GLA is geometrically ergodic, this paragraph quantifies 
the equilibrium error of GLA in preserving the BG measure.

\begin{lemma} \label{Deltahinfty} 
Assume \ref{sa1}, \ref{sa2}, and \ref{sa3}.  Then, there exist $C>0$ and $h_c>0$, such that
for all $h<h_c$, 
\[
\lim_{N \to \infty} \left|  \Delta_h^N( f ) \right|  \le  C  h^{p} ,
\]
and for all  $f \in L^2_{\mu_{\infty}}(\mathbb{R}^{2n})$ satisfying $| f(\boldsymbol{y}) | \le C_3 V(\boldsymbol{y})$ for all 
$\boldsymbol{y} \in \mathbb{R}^{2n}$.
\end{lemma}

\begin{proof}
Let $f \in L^2_{\mu}(\mathbb{R}^{2n})$ such that $| f(\boldsymbol{y}) | \le C_3 V(\boldsymbol{y})$
for all $\boldsymbol{y} \in \mathbb{R}^{2n}$.  The term $\E_{\mu} \E^{\boldsymbol{x}} \{ f( \boldsymbol{X}_N ) \} $ 
can be written as a telescoping sum:
\begin{align*}
 \E_{\mu} \E^{\boldsymbol{x}} \{ f( \boldsymbol{X}_N ) \} =  \mu(f) +   
 \sum_{k=1}^{N} \left( \E_{\mu}  \E^{\boldsymbol{x}} \{ f( \boldsymbol{X}_k ) \}  -    
 \E_{\mu} \E^{\boldsymbol{x}} \{ f( \boldsymbol{X}_{k-1} ) \}  \right) \text{.}
\end{align*}
By Lemma~\ref{Deltah1}, one can rewrite/reindex this sum as:
\begin{equation}
\Delta_h^N( f ) =  \int_{\mathbb{R}^{2n}}\sum_{k=0}^{N-1} \E^{\boldsymbol{x}} \left\{ f( \boldsymbol{X}_k ) \right\} \left( e^{ - \beta 
(H ( \theta_h^{-1}(\boldsymbol{x})) - H(\boldsymbol{x}) ) } - 1 \right)   \mu(d \boldsymbol{x}) \text{.}
\end{equation}
Since $\theta_h$ preserves Lebesgue measure, one can write this deviation as:
\begin{align} \label{DeltahN}
\Delta_h^N & ( f ) = \nonumber \\
& \int_{\mathbb{R}^{2n}}  \sum_{k=0}^{N-1}  \underset{\text{Deviation from Equilibrium}}{\underbrace{
\left( \E^{\boldsymbol{x}} \left\{ f( \boldsymbol{X}_k ) \right\}  - \mu_h(f) \right)  }} \cdot
\underset{\text{Energy Error of Variational Integrator}}{\underbrace{
\left( e^{ - \beta (H ( \theta_h^{-1}(\boldsymbol{x}) ) - H(\boldsymbol{x}) ) } - 1 \right)  }}   \mu(d \boldsymbol{x}) \text{.} 
\end{align}
From \eqref{DeltahN} it is clear that the equilibrium BG error is due to: 1) how fast GLA
converges to equilibrium and 2) the local accuracy with which $\theta_h$ represents the 
Hamiltonian function $H$.  The equality \eqref{DeltahN} is the crux of the proof, and 
what follows is an approach to bound $\Delta_h^N(f)$.

Since GLA is geometrically ergodic (cf.~Theorem~\ref{GLAgeometricergodicity}), one can 
bound $\Delta_h^N(f)$ from above by
\begin{equation*} 
\left|  \Delta_h^N( f ) \right|   \le    \left(  \sum_{k=0}^{N-1}   e^{-\lambda h k}  \right) C_3  \int_{\mathbb{R}^{2n}}  V( \boldsymbol{x}  )
\left|  e^{ - \beta (H ( \theta_h^{-1}(\boldsymbol{x}) ) - H(\boldsymbol{x}) ) } - 1 \right| \mu(d \boldsymbol{x})  \text{.}
\end{equation*}
Changing variables in the right-hand-side under the map $\theta_h$, one can rewrite this bound as,
\begin{equation*} 
\left|  \Delta_h^N( f ) \right|   \le    \left(  \sum_{k=0}^{N-1}   e^{-\lambda h k} \right) C_3  \int_{\mathbb{R}^{2n}}  V( \theta_h(\boldsymbol{x})  ) 
\left|  e^{ - \beta (H ( \theta_h(\boldsymbol{x}) )- H(\boldsymbol{x}) ) } - 1 \right| \mu(d \boldsymbol{x})  \text{.}
\end{equation*}
In the limit as $N \to \infty$, the right-hand-side of the above can be written in terms of the 
formula for the geometric series for $e^{-\lambda h}$:
\begin{equation} 
\lim_{N \to \infty} \left|  \Delta_h^N( f ) \right|   \le    \frac{C_3}{1 - e^{-\lambda h}}  \int_{\mathbb{R}^{2n}} V( \theta_h(\boldsymbol{x})  )  
\left|  e^{ - \beta (H ( \theta_h(\boldsymbol{x}) ) - H(\boldsymbol{x}) ) } - 1 \right| \mu(d \boldsymbol{x})  \text{.}
\end{equation}
Using the natural bound $| e^{x} - 1 | \le e^{| x |} - 1$ for all $x \in \mathbb{R}$, one can further 
bound $| \Delta_h^N(f) |$ by:
\begin{equation}  \label{DeltahNieq}
\lim_{N \to \infty} \left|  \Delta_h^N( f ) \right|   \le    \frac{C_3}{1 - e^{-\lambda h}} \int_{\mathbb{R}^{2n}} V( \theta_h(\boldsymbol{x})  ) 
\left( e^{ \beta | H ( \theta_h(\boldsymbol{x}) ) - H(\boldsymbol{x}) | } - 1 \right) \mu(d \boldsymbol{x})  \text{.}
 \end{equation}
 Introduce the exact flow $\vartheta_h$ of Hamilton's equations \eqref{HamiltonsEquations} into this bound,
 \begin{equation}  \label{DeltahNieq}
\lim_{N \to \infty} \left|  \Delta_h^N( f ) \right|   \le    \frac{C_3}{1 - e^{-\lambda h}} \int_{\mathbb{R}^{2n}}  V( \theta_h(\boldsymbol{x})  ) 
\left( e^{ \beta | H ( \theta_h(\boldsymbol{x}) )- H(\vartheta_h(\boldsymbol{x})) | } - 1 \right) \mu(d \boldsymbol{x})  \text{.}
 \end{equation}

Set $\boldsymbol{y}_0 = \theta_h(\boldsymbol{x})$ and $\boldsymbol{y}_1 = \vartheta_h(\boldsymbol{x})$.  
By the fundamental theorem of calculus,
\begin{equation*}
 H(\boldsymbol{y}_1) - H(\boldsymbol{y}_0) = 
 \int_0^1 \nabla H(\boldsymbol{y}_0 + s (\boldsymbol{y}_1 - \boldsymbol{y}_0)) 
 \cdot (\boldsymbol{y}_1 - \boldsymbol{y}_0) ds  \text{.}
\end{equation*}
Using condition {\bf U1} and the Cauchy-Schwartz inequality, it follows from the above that 
there exists $C > 0$ such that
\[
 |  H(\boldsymbol{y}_1) - H(\boldsymbol{y}_0) | \le 
 C (1+ | \boldsymbol{y}_1 | + | \boldsymbol{y}_0 | ) |  \boldsymbol{y}_1 - \boldsymbol{y}_0 |  \text{.}
 \]
 Another application of the condition {\bf U1} and {\bf V2} implies 
 there exists $C > 0$ such that
 \[
 |  H(\boldsymbol{y}_1) - H(\boldsymbol{y}_0) | \le 
 C (1+ | \boldsymbol{x} |^2 ) h^{p+1} \text{.}
 \]
 Therefore,
  \begin{equation}  \label{DeltahNieq}
\lim_{N \to \infty} \left|  \Delta_h^N( f ) \right|   \le    \frac{C_3}{1 - e^{-\lambda h}} \int_{\mathbb{R}^{2n}} V( \theta_h(\boldsymbol{x})  ) 
\left( e^{ \beta K (1+ | \boldsymbol{x} |^2 ) h^{p+1} } - 1 \right) \mu(d \boldsymbol{x})  \text{.}
 \end{equation}
 Now we show how the the factor $V( \theta_h(\boldsymbol{x})  ) $ above is handled.

 Since the Lyapunov function is quadratically bounded, the variational integrator satisfies {\bf V2}, and the
 Hamiltonian vector field is uniformly Lipschitz by condition {\bf U1}, there exists $C > 0$ such that
   \begin{equation} 
\lim_{N \to \infty} \left|  \Delta_h^N( f ) \right|   \le    \frac{C}{1 - e^{-\lambda h}} \int_{\mathbb{R}^{2n}}  (1 + |\boldsymbol{x} |^2  ) 
\left( e^{ \beta K (1+ | \boldsymbol{x} |^2 ) h^{p+1} } - 1 \right) \mu(d \boldsymbol{x})  \text{.}
 \end{equation}
By condition {\bf U2} the total energy is quadratically bounded from below. Consequently one can 
bound $e^{- \beta H(\boldsymbol{x})}$ by $e^{- \beta D (1 + | \boldsymbol{x} |^2 ) }$
for some constant $D > 0$.   Thus,
   \begin{align*} 
& \lim_{N \to \infty} \left|  \Delta_h^N( f ) \right|   \le \\
 & \qquad   \frac{C}{1 - e^{-\lambda h}} \int_{\mathbb{R}^{2n}}  (1 + |\boldsymbol{x} |^2  ) 
\left( e^{ \beta K (1+ | \boldsymbol{x} |^2 ) h^{p+1} } - 1 \right) e^{- \beta D (1 + | \boldsymbol{x} |^2 ) } d \boldsymbol{x}  \text{.}
 \end{align*}
When $h < h_c= (D/K)^{1/(p+1)}$ the above integral is finite and one obtains the desired error estimate.
 \end{proof}

A simple application of  Theorem~\ref{GLABGaccuracy} implies an error estimate for $\mu_h$.     
For this purpose we introduce the total variation between measures $\mu$ and $\nu$:
\[
| \mu - \nu |_{TV}  =  \sup_{| f | \le 1} \left| \int_{\mathbb{R}^{2n}}
f(\boldsymbol{x}) (\mu(d \boldsymbol{x}) -  \nu(d \boldsymbol{x})) \right| \text{.}
\]
Since $\tilde{M}( \boldsymbol{y}) \ge 1$ for all $ \boldsymbol{y} \in \mathbb{R}^{2n}$,
Theorem~\ref{GLABGaccuracy} applies for all
$f \in L^2_{\mu}(\mathbb{R}^{2n})$ such that  $| f(\boldsymbol{y}) | \le 1$ for all $\boldsymbol{y} \in \mathbb{R}^{2n}$.  
The TV norm can be written as:
\begin{align*}
&| \mu - \mu_h |_{TV} =   \\
& \qquad  \sup_{| f | \le 1} \left|  \int_{\mathbb{R}^{2n}} f d \mu - \E_{\mu} \E^{\boldsymbol{x}} \{ f( \boldsymbol{X}_N ) \}  
+ \E_{\mu} \E^{\boldsymbol{x}} \{ f( \boldsymbol{X}_N ) \}  - \int_{\mathbb{R}^{2n}} f d\mu_h \right| 
\end{align*}
By the triangle inequality,
\begin{align}
| \mu - \mu_h |_{TV} \le  \sup_{| f | \le 1} \left| \Delta_h^N(f) \right| 
+ \sup_{| f | \le 1} \left| \E_{\mu} \E^{\boldsymbol{x}} \{ f( \boldsymbol{X}_N ) \} -\mu_h(f) \right| \text{.} \label{eq:tvprelimit}
\end{align}
However, under the hypotheses of the theorem, $GLA$ is geometrically ergodic 
with respect to $\mu_h$ and hence,
\begin{equation} \label{eq:decays}
\lim_{N \to \infty} \sup_{| f | \le 1}  
\left| \E_{\mu} \E^{\boldsymbol{x}} \{ f( \boldsymbol{X}_N ) \} -\mu_h(f) \right| \to 0 
\end{equation}
and,
\begin{align} \label{tvestimate}
| \mu - \mu_h |_{TV} \le  \lim_{N \to \infty} \sup_{| f | \le 1} \left| \Delta_h^N(f) \right|  \text{.}
\end{align}
Lemma~\ref{Deltahinfty} can now be invoked to obtain from \eqref{tvestimate} an 
upper bound for the TV distance between $\mu$ and $\mu_h$.
This concludes the proof of Theorem~\ref{GLABGaccuracy} which we restate:

\begin{thm23}[Long-Run Accuracy]
Assume \ref{sa1}, \ref{sa2}, and \ref{sa3}.    Let $\mu_h$ denote the discrete invariant measure of GLA. 
Then, there exist $C>0$ and $h_c>0$, such that for all $h<h_c$, 
\[
  | \mu - \mu_h |_{TV} \le C h^{p}  \text{.}
\]
\end{thm23}

In summary, the preceding analysis showed the TV error estimate in Theorem~\ref{GLABGaccuracy} 
relies on GLA's variational integrator $\theta_h$ being volume-preserving and $pth$-order accurate, 
the Ornstein-Uhlenbeck map $\psi_h$ exactly preserving the Boltzmann-Gibbs measure, 
and GLA being geometrically ergodic.   To establish the latter, we used the strategy adopted in 
\cite{MaStHi2002} which relates pathwise convergence of a discretization of an SDE to 
geometric ergodicity of the discretization.  This strategy requires the potential force is uniformly 
Lipschitz.


\section{Validation} \label{Validation}

This section tests three different instances of GLA on a variety of simple
mechanical systems governed by Langevin equations.  The purpose of this section is to 
confirm the error estimates provided in the paper.

Let $h$ be a fixed step size and $\xi_k \sim \mathcal{N}(0,1)$ for $k \in \mathbb{N}$.
The following update scheme is obtained by composing the explicit first-order, symplectic Euler 
method with $\psi_h$:
\begin{equation} \label{eq:bgse}
\begin{cases}
\begin{array}{rcl}
\hat{p}_k &=& e^{-\gamma h} p_k +  \sqrt{\frac{1-e^{-2 \gamma h}}{\beta}} \xi_k \text{,} \\
q_{k+1} &=& q_k + h  \hat{p}_k \text{,} \\
p_{k+1} &=&  \hat{p}_k - h  \frac{\partial U}{\partial q}(q_{k+1}) \text{,}
\end{array}
\end{cases}
\end{equation}
for $k \in \mathbb{N}$. The following integrator is obtained by composing the second-order 
accurate explicit, symmetric, symplectic St\"{o}rmer-Verlet
method  with $\psi_h$:
\begin{equation} \label{eq:bgsv}
\begin{cases}
\begin{array}{rcl}
\hat{p}_k &=& e^{-\gamma h} p_k +  \sqrt{\frac{1-e^{-2 \gamma h}}{\beta}} \xi_k \text{,} \\
P_k^{1/2} &=& \hat{p}_k - \frac{h}{2} \frac{\partial U}{\partial q}(q_k)  \text{,} \\
q_{k+1} &=& q_k + h  P_k^{1/2}   \text{,} \\
p_{k+1} &=& P_k^{1/2} - \frac{h}{2}  \frac{\partial U}{\partial q}(q_{k+1})   \text{,}
\end{array}
\end{cases}
\end{equation}
for $k \in \mathbb{N}$.  The following integrator is obtained by composing a fourth-order 
accurate explicit, symmetric, symplectic method due
to~F.~Neri (see, e.g., \cite{Yo1990}) with $\psi_h$:
\begin{equation} \label{eq:bgne}
\begin{cases}
\begin{array}{rcl}
Q_1 &=& q_k \text{,} \\
P_1 &=& e^{-\gamma h} p_k +  \sqrt{\frac{1-e^{-2 \gamma h}}{\beta}} \xi_k \text{,}  \\
\end{array}  \\
\begin{cases}
 \begin{array}{rcl}
P_{i+1} &=& P_i - c_i h \frac{\partial U}{\partial q}(Q_i)  \text{,}  \\
Q_{i+1} &=& Q_i + d_i h P_{i+1} \text{,}
\end{array}~~~   i =1,..., 4, 
\end{cases} \\
\begin{array}{rcl}
q_{k+1} &=& Q_{5} \text{,} \\
p_{k+1} &=& P_{5}   \text{,}
\end{array}
\end{cases}
\end{equation}
for $k \in \mathbb{N}$, and where we have introduced the following constants: 
\begin{align*}
c_1 = c_4 = \frac{1}{2 (2 -2^{1/3})}, &~~~ c_2 = c_3 = \frac{1- 2^{1/3}}{2 ( 2- 2^{1/3})}, ~~\\
d_1 = d_3 = \frac{1}{2 -2^{1/3}}, &~~~ d_2 = \frac{-2^{1/3}}{2- 2^{1/3}}, d_4 = 0 \text{.}
\end{align*}
The purpose of this fourth-order symplectic integrator is for validation.  For ``optimal'' fourth and 
fifth-order accurate symplectic integrators 
that minimize the error in the Hamiltonian, the reader is referred to \cite{McAt1992}.

We will show that despite the fact that (\ref{eq:bgsv}) and (\ref{eq:bgne}) are only first-order 
pathwise convergent according to Theorem~\ref{GLAaccuracy}, they approximate ensemble averages of 
$\mu$-integrable functions that satisfy $| f(q,p) | \le M(q,p)$ for all $(q,p) \in \mathbb{R}^{2n}$ to within second and 
fourth-order accuracy, respectively.     This is consistent with Theorem~\ref{GLABGaccuracy}.

\paragraph{Linear Oscillator}

This section follows the analysis of numerical methods for linear oscillators governed by 
Langevin equations developed in \cite{MiTr2004, BuLeLy2007}.
The governing equations for a linear oscillator of unit mass at uniform temperature $1/\beta$ 
are given explicitly by evaluating \eqref{InertialLangevin} at $U(q) = q^2/2$:
\begin{equation}
\begin{cases}
\begin{array}{rcl}
d q &=& p dt  \text{,}  \\
d p &= & -q  dt - \gamma p dt + \sqrt{2 \beta^{-1} \gamma} d W  \text{.}
\end{array}
\end{cases}
\end{equation}
The resulting process is Gaussian with stationary distribution given by the BG distribution:
\[
P_{\infty}(q,p) = Z^{-1} \exp\left( - \beta \left( \frac{ p^2}{2} + \frac{q^2}{2} \right) \right)
\]
and with
\[
\mu(q^2) = \lim_{t \to \infty} \E \{ q_t^2 \} = 1/\beta, ~~~ \mu(p^2) = \lim_{t \to \infty} \E \{ p_t^2 \} = 
1/\beta,~~~ \kappa(q p) = \lim_{t \to \infty} \E\{ q_t p_t \} = 0 \text{.}
\]

The stationary distribution of the geometric Langevin integrators (\ref{eq:bgse})-(\ref{eq:bgne}) 
is also Gaussian with equilibrium distribution of the form:
\[
P_h (q, p) = \frac{1}{2 \pi | \Sigma^{-1} | } \exp\left( -\frac{1}{2} \begin{pmatrix} q  &  p 
\end{pmatrix} \Sigma^{-1} \begin{pmatrix} q \\ p \end{pmatrix} \right)
\]
where 
\[
\Sigma = \begin{bmatrix} \sigma_q^2 & \kappa \\ \kappa & \sigma_p^2 \end{bmatrix}, ~~~ 
\sigma_q^2 = \lim_{n \to \infty} \E \{ q_n^2 \} , ~~~ \sigma_p^2 = \lim_{n \to \infty} \E \{ p_n^2 \} , 
~~~ \kappa = \lim_{n \to \infty} \E \{ q_n p_n \} \text{.}
\]
This stationary correlation matrix can be explicitly determined.   
For (\ref{eq:bgse}) its entries are given by:
\begin{align*}
\sigma_q^2 &=  \frac{ \left( 1+ e^{\gamma h} \right)^2 }{ \left(2 + 2 e^{\gamma h} - h^2 \right) \beta} = \frac{1}
{\beta} + \mathcal{O}(h) \\
\sigma_p^2 &= \frac{2 + 2 e^{\gamma h} - h^2 + e^{2 \gamma h} h^2}{\left( 2 + 2 e^{\gamma h} - h^2\right) \beta} = 
\frac{1}{\beta} + \mathcal{O}(h^2) \\
\gamma &=   - \frac{e^{\gamma h} \left(1 + e^{\gamma h} \right) h}{\left( 2 + 2 e^{\gamma h} - h^2\right) \beta}  = 
\mathcal{O}(h)
\end{align*}
Observe that the cumulative error (\ref{eq:bgse}) makes is of $\mathcal{O}(h)$, i.e.,
\[
| \sigma_q^2 - \mu(q^2 ) | + 
| \sigma_p^2 - \mu(p)^2 ) |  +
| \kappa - \mu(q p ) |  \le   \mathcal{O}(h) \text{.}
\]
Whereas for (\ref{eq:bgsv}) its entries are given by:
\begin{align*}
\sigma_q^2 &=  \frac{4}{\beta ( 4 - h^2  )} = \frac{1}{\beta} +  \frac{h^2}{4 \beta} + \mathcal{O}
(h^4) \\
\sigma_p^2 &= \frac{1}{\beta}  \\
\kappa &=   0
\end{align*}
and its cumulative error is of $\mathcal{O}(h^2)$, i.e., 
\[
| \sigma_q^2 - \mu(q^2 ) | + 
| \sigma_p^2 - \mu(p)^2 ) |  +
| \kappa - \mu(q p ) |  \le   \mathcal{O}(h^2) \text{.}
\]
For (\ref{eq:bgne}) its entries are given by:
\begin{align*}
\sigma_q^2 &=  \frac{1}{\beta} + \frac{\left(-4-3 \times \sqrt[3]{2}-2 \times 2^{2/3}\right)
   h^4}{144 \beta } + O(h^5) \\
\sigma_p^2 &= \frac{1}{\beta}  \\
\kappa &=   0
\end{align*}
and its cumulative error is of $\mathcal{O}(h^4)$, i.e., 
\[
| \sigma_q^2 - \mu(q^2 ) | + 
| \sigma_p^2 - \mu(p)^2 ) |  +
| \kappa - \mu(q p ) |   \le   \mathcal{O}(h^4) \text{.}
\]

Finally, consider the exact splitting applied to the linear oscillator at uniform temperature.  
Hamilton's equations for a linear  oscillator are:
\[
\begin{bmatrix} \dot{q} \\ \dot{p} \end{bmatrix}(t) = \begin{bmatrix} 0 & 1 \\ -1 & 0 \end{bmatrix} 
\begin{bmatrix} q \\ p \end{bmatrix}(t), ~~~  \begin{bmatrix} q \\ p \end{bmatrix}(0) =  
\begin{bmatrix} q_0 \\ p_0 \end{bmatrix} \text{,}
\]
with explicit solution given by:
\[
\begin{bmatrix} q \\ p \end{bmatrix}(t) =   \begin{bmatrix} \cos(t) & \sin(t) \\ -\sin(t) & \cos(t) 
\end{bmatrix}  \begin{bmatrix} q_0 \\ p_0 \end{bmatrix} \text{.}
\]
Thus, the exact splitting update is given by:
\[
\begin{bmatrix} q_1 \\ p_1 \end{bmatrix} =   \begin{bmatrix} \cos(h) & \sin(h) \\ -\sin(h) & \cos(h) 
\end{bmatrix} \begin{bmatrix} q_0 \\ p_0 \end{bmatrix} + 
 \sqrt{\frac{e^{2 \gamma h} - 1}{\beta}} \begin{bmatrix} \sin(h)  \\ \cos(h) \end{bmatrix}  \xi_0  \text{.}
\]
In this situation one can show there is no error made in the stationary correlation matrix.  This
follows from the fact that the exact solution of Hamilton's equations is volume and energy 
preserving.

\paragraph{Nonglobally Lipschitz, Nonlinear Oscillator}

The theory in this paper does not apply to this example since the potential force is nonglobally
Lipschitz.  With a nonglobally Lipschitz potential force, for any $h>0$ there will exist regions in phase space 
where the Lipschitz constant of the potential force is beyond the linear stability threshold of an
explicit variational integrator $\theta_h$.  Hence, a GLA based on an explicit variational integrator will be 
stochastically unstable; transient, to be precise.  However, for the 
step-sizes and variational integrators employed, and for the duration of the numerical experiments, 
discrete orbits of GLA seem to be confined to a compact region of phase space where the 
variational integrator $\theta_h$ is linearly stable and Monte Carlo estimates are consistent 
with the error estimates in the paper.

The governing equations for a cubic oscillator of unit mass at uniform
temperature $1/\beta$ are given explicitly by evaluating
\eqref{InertialLangevin} at $U(q) = q^4/4 - q^2/2$:
\begin{equation}
\begin{cases}
\begin{array}{rcl}
d q &=& p dt  \text{,}  \\
d p &= & (q - q^3) dt - \gamma p dt + \sqrt{2 \beta^{-1} \gamma} d W  \text{.}
\end{array}
\end{cases}
\end{equation}
The resulting potential force is only locally Lipschitz.

The estimates shown earlier predict that
\[
| \mu(q^2) - \mu_h(q^2) |  \le  \mathcal{O}(h^p)
\]
where $p$ is the order of accuracy of $\theta_h$.    Hence, one expects near fourth-order 
accuracy for \eqref{eq:bgne},
near second-order accuracy for \eqref{eq:bgsv} and first-order accuracy for \eqref{eq:bgse} as 
shown in 
table~\ref{tab:coerrorstats}.  The tests will apply (\ref{eq:bgse})-(\ref{eq:bgne}) to estimate  
\[
\lim_{t \to \infty} \E \{ q_t^2 \} = \mu(q^2) =  \frac{\int_{-\infty}^{\infty} q^2 e^{- \beta U(q)} dq}
{ \int_{-\infty}^{\infty} e^{- \beta U(q)} dq}
\]
by empirical averages of the form  
\[
I^{h, N} := \frac{1}{N} \left( \sum_{i=1}^N q_i^{2} \right) \text{.}
\]  
As nicely discussed in \cite{Ta2002}, in addition to the discretization error $| \mu(q^2) - 
\mu_h(q^2) | $ one has to cope with the statistical error arising from the 
time-average being finite, i.e., $I^{h, N} \approx \mu_h(q^2)$.  
The computations were performed with $\gamma = 1$ and an inverse temperature 
value of $ \beta = 2$.

\begin{table}
\begin{center}
\begin{tabular}{ccccc}
Time-Step & Number of Steps    &   \eqref{eq:bgse} & \eqref{eq:bgsv}  & \eqref{eq:bgne}  \\
 \hline 
 \hline
h  & N  &   3.11e-02  &      8.03e-03    &   1.45e-02  \\ 
 \hline
h/2  & 2 N &  1.49e-02     &   1.94e-03  & 9.80e-04  \\
 \hline
h/4 & 4 N &   7.42e-03     &   4.83e-04     & 7.35e-05 \\
 \hline
 h/8  & 8 N &  3.74e-03   &   1.29e-04      &  5.79e-06 \\
\end{tabular}
\end{center}
\caption{ 
The table estimates $\left| \mu_h(q^2) -   \mu(q^2)  \right|$ using empirical time-averages with 
$N=40 \times 10^9$ steps and $h=0.4$ with GLA as determined by (\ref{eq:bgse})-(\ref{eq:bgne}).   
For subsequent rows the time-steps are halved and the number of steps doubled, so that the time-interval 
of integration is fixed for all experiments. The results show that as the time-steps are halved the difference 
decreases linearly for \eqref{eq:bgse}, nearly quadratically for \eqref{eq:bgsv}, and  nearly quartically 
for \eqref{eq:bgne}.  These results are consistent with the error estimates in the paper.  }
\label{tab:coerrorstats}
\end{table}


\section{Conclusion}

The analysis in this paper represents a first step towards a deeper analysis of GLA for 
molecular systems.    In this paper we make assumptions on the Hamiltonian that ensure the 
solution to inertial Langevin and GLA are geometrically ergodic.  In particular, we assume the 
Hamiltonian vector field is uniformly Lipschitz and the Hamiltonian is coercive.  These 
hypotheses are sufficient to ensure GLA is geometrically ergodic whenever the solution process
is. In particular, the former hypothesis is important to ensure GLA is stochastically stable  
\cite{MeTw2009}.   If GLA's underlying variational integrator is not globally linearly 
stable, one can show GLA defines a transient Markov chain.    Still one can use GLA 
as proposal step within a Metropolis-Hastings algorithm to obtain a stochastically stable 
Metropolis-Adjusted Geometric Langevin Algorithm  (MAGLA).  A numerical analysis of 
MAGLA including pathwise convergence can be found in \cite{BoVa2009A}.

A closer inspection of the proof of Theorem~\ref{GLABGaccuracy} reveals that the estimate
relies on the following important ingredients:
\begin{enumerate}
\item GLA is geometrically ergodic with respect to a probability measure $\mu_h$;
\item the variational integrator is Lebesgue-measure preserving;
\item the Ornstein-Uhlenbeck flow preserves $\mu$; and,
\item the local energy error of the variational integrator is $(p+1)th$-order accurate.
\end{enumerate} 
Therefore, we stress that the result holds under more general conditions.  The main point being:
\begin{quote}
{\em
If GLA is geometrically ergodic with respect to a unique invariant measure, the error in sampling the 
invariant measure of the SDE is determined by the energy error in GLA's variational integrator.}
\end{quote}


\section{Appendix}

\subsection{Single-Step Error}

\begin{lemma} \label{singlesteperror}
Assume \ref{sa1} and \ref{sa2}.
For $h$ small enough, there exists a $C>0$ such that 
\begin{equation} 
| \E^{\boldsymbol{x}} \{   \boldsymbol{Y}(h)  - \boldsymbol{Z}_{1}  \} | 
\le C \left( 1+ | \boldsymbol{x} |^2 \right)^{1/2}  h^{2} 
\end{equation}
and
\begin{equation} 
\left( \E^{\boldsymbol{x}} \{ | \boldsymbol{Y}(h)  - \boldsymbol{Z}_{1} |^2 \} \right)^{1/2} 
\le C \left( 1+ | \boldsymbol{x} |^2 \right)^{1/2}  h^{3/2}  \text{.}
\end{equation}
\end{lemma}

\begin{proof}
Write $\mathbf{Y}(t)=(\boldsymbol{Q}(t),\boldsymbol{P}(t))$ where $\boldsymbol{Q}(t)$ 
and  $\boldsymbol{P}(t)$ represent the instantaneous configuration and 
momentum of the system, respectively.  In terms of which write the 
SDE \eqref{InertialLangevin} as:
\begin{equation} \label{langevin}
\begin{cases}
d \boldsymbol{Q} / dt &= \boldsymbol{M}^{-1} \boldsymbol{P}   \\
d \boldsymbol{P} &= - \nabla U( \boldsymbol{Q} ) dt  
	- \gamma  \boldsymbol{M}^{-1} \boldsymbol{P} dt + \sqrt{2 \gamma \beta^{-1}} d \boldsymbol{W}
	\end{cases}
\end{equation}
$\boldsymbol{Q}(0) = \boldsymbol{Q}_0$ and $\boldsymbol{P}(0) = \boldsymbol{P}_0$.
It will be useful to write out the solution of  \eqref{langevin}.
For this purpose integrate \eqref{langevin} to obtain: \begin{align} 
& \boldsymbol{Q}(h) = \boldsymbol{Q}_0
	+ h \boldsymbol{M}^{-1} \boldsymbol{P}_0
	+ \int_0^h \boldsymbol{M}^{-1} [ - \nabla U(\boldsymbol{Q}(s)) - \gamma \boldsymbol{M}^{-1} \boldsymbol{P}(s) ] ( h - s) ds \nonumber \\
& \qquad + \sqrt{2 \gamma \beta^{-1}}  \int_{0}^{h} (h - s) \boldsymbol{M}^{-1} 
	d \boldsymbol{W}(s)  \label{Qcontinuous}
\end{align}
and
\begin{align} 
& \boldsymbol{P}(h) = e^{-\gamma \boldsymbol{M}^{-1} h} 
	\boldsymbol{P}_0 - h  \nabla U( \boldsymbol{Q}_0)  - \int_{0}^{h} (h-s) \frac{\partial^2 U}{\partial \boldsymbol{q}^2}(\boldsymbol{Q}(s)) \cdot \boldsymbol{M}^{-1} \boldsymbol{P}(s) ds \nonumber \\
& \qquad + \int_0^h (\boldsymbol{I} -  e^{- \gamma \boldsymbol{M}^{-1} (h-s)} ) \nabla U(\boldsymbol{Q}(s)) ds + \boldsymbol{\eta} \label{Pcontinuous}
\end{align}
where we have introduced: \[
\boldsymbol{\eta} = \sqrt{2 \gamma \beta^{-1}} \int_0^h e^{-\gamma \boldsymbol{M}^{-1} (h - s) } d \boldsymbol{W}(s) \text{.}
\]

Write $\mathbf{Z}(t)=(\hat{\boldsymbol{Q}}(t),\hat{\boldsymbol{P}}(t))$ where $\hat{\boldsymbol{Q}}(t)$ 
and  $\hat{\boldsymbol{P}}(t)$ represent the instantaneous configuration and 
momentum of the exact splitting, respectively.  
The exact splitting after a single step solves \begin{equation} \label{hamiltonseqns}
\begin{cases}
d \hat{\boldsymbol{Q}}/dt &= \boldsymbol{M}^{-1} \hat{\boldsymbol{P}} \\
d \hat{\boldsymbol{P}}/dt &= - \nabla U(\hat{\boldsymbol{Q}}) 
\end{cases}
\end{equation}
where $\hat{\boldsymbol{Q}}(0) = \boldsymbol{Q}_0$ and $\hat{\boldsymbol{P}}(0) = e^{-\gamma \boldsymbol{M}^{-1} h} \boldsymbol{P}_0 + \boldsymbol{\eta}$.  Integrating \eqref{hamiltonseqns} 
yields,
 \begin{align} 
& \hat{\boldsymbol{Q}}(h) = \boldsymbol{Q}_0
	+ h \boldsymbol{M}^{-1}  e^{-\gamma \boldsymbol{M}^{-1} h} \boldsymbol{P}_0
	- \int_0^h \boldsymbol{M}^{-1}  \nabla U(\hat{\boldsymbol{Q}}(s)) ( h - s) ds  + h \boldsymbol{M}^{-1} \boldsymbol{\eta}  \label{Qdiscrete}
\end{align}
and
\begin{align} 
& \hat{\boldsymbol{P}}(h) = e^{-\gamma \boldsymbol{M}^{-1} h} 
	\boldsymbol{P}_0 - h  \nabla U( \boldsymbol{Q}_0)  \nonumber \\
& \qquad  - \int_{0}^{h} (h-s) \frac{\partial^2 U}{\partial \boldsymbol{q}^2}(\hat{\boldsymbol{Q}}(s)) \cdot \boldsymbol{M}^{-1} \hat{\boldsymbol{P}}(s) ds 
	+ \boldsymbol{\eta} \label{Pdiscrete} \text{.}
\end{align}

To obtain the mean-squared and mean error
estimates we will use the following bounds on the second moment of the continuous solution and the
exact splitting. Namely, for all $t \in [0,h]$, there exists a $C>0$ such that 
\begin{equation} \label{momentbounds}
\E^{\boldsymbol{x}} \left\{ |  \boldsymbol{Z}(t)  |^2 \right\} 
\vee
\E^{\boldsymbol{x}} \left\{ |  \boldsymbol{Y}(t)  |^2 \right\} \le C (1 + | \boldsymbol{x} |^2) 
\end{equation}
where $\boldsymbol{x} = (\boldsymbol{Q}_0, \boldsymbol{P}_0)$.
 We will prove this estimate for the exact splitting, and omit the proof 
for the continuous solution since it is very similar.  Let $\hat{\boldsymbol{x}} = (\hat{\boldsymbol{Q}}(0), \hat{\boldsymbol{P}}(0))$. By Taylor's formula, 
\begin{align*}
 & |  \boldsymbol{Z}(t)  |^2 =  | \hat{\boldsymbol{x}} |^2 + 2 \int_0^t \left\langle \hat{\boldsymbol{Q}}(s), \boldsymbol{M}^{-1} \hat{\boldsymbol{P}}(s) \right\rangle ds +  2 \int_0^t \left\langle \hat{\boldsymbol{P}}(s), - \nabla U( \hat{\boldsymbol{Q}}(s) ) \right\rangle ds 
\end{align*}
By Young's inequality,
\begin{align*}
 & |  \boldsymbol{Z}(t)  |^2 \le  | \hat{\boldsymbol{x}} |^2  \\
& \qquad +  \int_0^t ( | \hat{\boldsymbol{Q}}(s) | ^2 + | \boldsymbol{M}^{-1} \hat{\boldsymbol{P}}(s) |^2 ) ds 
+   \int_0^t ( | \hat{\boldsymbol{P}}(s) | ^2  + | \nabla U( \hat{\boldsymbol{Q}}(s) )|^2 ) ds 
\end{align*}
The uniform Lipschitz condition {\bf U1} implies a linear growth condition on the potential force.   Hence,
there exists a constant $C>0$ such that
\begin{align*}
  |  \boldsymbol{Z}(t)  |^2 \le   | \hat{\boldsymbol{x}} |^2
  +  C \int_0^t  | \hat{\boldsymbol{Z}}(s) | ^2 ds 
\end{align*}
By Gronwall's lemma it follows that,
\begin{align*}
  |  \boldsymbol{Z}(t)  |^2 \le   | \hat{\boldsymbol{x}} |^2  e^{C h} 
\end{align*}
for $t \le h$.   Hence, for $h$ small enough we obtain the desired bound on the 
second moment of the exact splitting.

The difference between \eqref{Qcontinuous} and \eqref{Qdiscrete} is,
\begin{align}
& \boldsymbol{Q}(h) - \hat{\boldsymbol{Q}}(h) = \nonumber \\ 
& \qquad h \boldsymbol{M}^{-1} ( \boldsymbol{I} - e^{-\gamma \boldsymbol{M}^{-1} h } ) \boldsymbol{P}_0  \nonumber \\ 
& \qquad + \int_0^h \boldsymbol{M}^{-1} [ \nabla U( \hat{\boldsymbol{Q}}(s) )  - \nabla U( \boldsymbol{Q}(s) ) ] (h - s) ds \nonumber \\
& \qquad + \int_0^h \boldsymbol{M}^{-1} [ - \gamma \boldsymbol{P}(s) ds + \sqrt{2 \gamma \beta^{-1}} d \boldsymbol{W}(s) ] ( h - s) - h \boldsymbol{M}^{-1} \boldsymbol{\eta} \label{Qerror}
\end{align}
Likewise, the difference between \eqref{Pcontinuous} and \eqref{Pdiscrete} is,
\begin{align}
& \boldsymbol{P}(h) - \hat{\boldsymbol{P}}(h) = \nonumber \\
& \qquad \int_0^h  (h-s) \left[ \frac{\partial^2 U}{\partial \boldsymbol{q}^2}(\hat{\boldsymbol{Q}}(s)) \cdot \boldsymbol{M}^{-1} \hat{\boldsymbol{P}}(s) - \frac{\partial^2 U}{\partial \boldsymbol{q}^2}(\boldsymbol{Q}(s)) \cdot \boldsymbol{M}^{-1} \boldsymbol{P}(s) \right]  ds \nonumber \\
& \qquad + \int_0^h (e^{-\gamma \boldsymbol{M}^{-1} (h - s) } - \boldsymbol{I} ) \nabla U( \boldsymbol{Q}(s) ) ds
\label{Perror}
\end{align}
From \eqref{Qerror} and \eqref{Perror}, it is clear that
the leading term of the expectation of these differences is $O(h^2)$ and the leading term in the mean-squared expectation of the
differences is $O(h^{3/2})$.  To bound these terms one needs the bounds on the second
moments of the solutions and the exact splitting provided in \eqref{momentbounds}.  
To enable estimation of \eqref{Perror} one needs control of the Hessian of $U$.
The assumption of smoothness on $U$ and the uniform Lipschitz condition {\bf U1} on the potential force
provide this control.  In particular, since a differentiable function is Lipschitz continuous if and only if it has bounded differential, 
the Frobenius norm of the Hessian of $U$ is bounded by the Lipschitz constant of the potential force.

\end{proof}

\subsection{Variational Integrators}

Let $L: \mathbb{R}^{2n} \to \mathbb{R}$ denote the Lagrangian obtained from the 
Legendre transform of the Hamiltonian $H$, and given by:
\[
L(\boldsymbol{q}, \boldsymbol{v}) = 
\frac{1}{2} \boldsymbol{v}^T \boldsymbol{M}  \boldsymbol{v} - U(\boldsymbol{q}) \text{.}
\]
A variational integrator is defined by a discrete Lagrangian $L_d : \mathbb{R}^n \times \mathbb{R}^n
\times \mathbb{R}^+ \to \mathbb{R}$ which is an approximation to the so-called {\em exact discrete 
Lagrangian} which is defined as:
\[
L_d^E(\boldsymbol{q}_0, \boldsymbol{q}_1, h) = \int_0^h L( \boldsymbol{Q}, \dot{\boldsymbol{Q}} ) dt
\]
where $\boldsymbol{Q}(t)$ solves the Euler-Lagrange equations for the Lagrangian $L$ 
with endpoint  conditions $\boldsymbol{Q}(0) = \boldsymbol{q}_0$ and 
$\boldsymbol{Q}(h) = \boldsymbol{q}_1$.

By passing to the Hamiltonian description, a discrete Lagrangian determines 
a symplectic integrator on $\mathbb{R}^{2n}$ as follows.  
Given $(\boldsymbol{q}_0,\boldsymbol{p}_0) \in \mathbb{R}^{2n}$, 
a variational integrator defines an update 
$(\boldsymbol{q}_1, \boldsymbol{p}_1) \in \mathbb{R}^{2n}$  
by the following system of equations:
\begin{equation} \label{del}
\begin{cases}
\begin{array}{rcl}
\boldsymbol{p}_0 &=& -D_1 L_d(\boldsymbol{q}_0, \boldsymbol{q}_1,h) \text{,} \\
\boldsymbol{p}_1 &=& D_2 L_d(\boldsymbol{q}_0, \boldsymbol{q}_1,h)  \text{.}
\end{array}
\end{cases}
\end{equation}
Denote this map by $\theta_h : \mathbb{R}^{2n} \to  \mathbb{R}^{2n}$, i.e., 
\[
\theta_h: ~~ (\boldsymbol{q}_0,\boldsymbol{p}_0) \mapsto (\boldsymbol{q}_1, \boldsymbol{p}_1)  \text{,}
\] 
where $(\boldsymbol{q}_1, \boldsymbol{p}_1)$ solve \eqref{del}.
One can show that $\theta_h$ preserves the canonical symplectic form 
on $\mathbb{R}^{2n}$, and hence, is Lebesgue measure preserving \cite{MaWe2001}.
By appropriately constructing $L_d$, the map $\theta_h$ can define an approximation 
to the flow of Hamilton's equations for the Hamiltonian $H$  \eqref{HamiltonsEquations}.   
Hyperregularity of the discrete Lagrangian means for all $h>0$
\begin{equation} \label{HyperRegularity}
| \det D_{12} L_d(\boldsymbol{q}_0, \boldsymbol{q}_1, h) | > 0, 
~~ \forall~ \boldsymbol{q}_0, \boldsymbol{q}_1 \in Q \times Q \text{.}
\end{equation}

\bibliography{nawaf}{}
\bibliographystyle{amsplain}

 \end{document}